\documentclass[preprint,3p,11pt,authoryear]{elsarticle}

\makeatletter
\def\ps@pprintTitle{%
 \let\@oddhead\@empty
 \let\@evenhead\@empty
 \def\@oddfoot{\centerline{\thepage}}%
 \let\@evenfoot\@oddfoot}
\makeatother

\usepackage{amssymb,amsthm,mathtools}
\usepackage{lineno} 

\usepackage{enumerate,natbib,caption,graphicx,subcaption,xcolor,geometry}
\usepackage{appendix} 
\usepackage{mathrsfs} 
\usepackage{dsfont} 
\usepackage{cancel} 

\usepackage[colorlinks]{hyperref}

\newtheorem{theorem}{Theorem}[section]

\newtheorem{corollary}[theorem]{Corollary}
\newtheorem{lemma}[theorem]{Lemma}
\newtheorem{remark}[theorem]{Remark}

\newtheorem*{notation}{Notation}

\makeatletter \@addtoreset{equation}{section} \makeatother

\newcommand{\N}{\mathbb{N}}
\newcommand{\R}{\mathbb{R}}

\newcommand{\QQ}{\mathbb{Q}}
\newcommand{\PP}{\mathbb{P}}
\newcommand{\EE}{\mathbb{E}}

\newcommand{\VV}{\mathbb{V}\mathrm{ar}}

\newcommand{\OO}{\mathcal O}
\newcommand{\oo}{\mathrm{o}}
\newcommand{\leqdef}{\vcentcolon=}

\newcommand{\rd}{{\rm d}}
\newcommand{\ind}{\mathds{1}}

\allowdisplaybreaks

\begin{document}

    \vspace{-4cm}
    \noindent
    \fbox{
    \begin{minipage}{36em}
        \small
        This manuscript was accepted for publication in Statistics (Taylor \& Francis).
        This version {\it may differ} from the published version (\href{https://doi.org/10.1080/02331888.2022.2084544}{doi:10.1080/02331888.2022.2084544}) in typographic details.
    \end{minipage}
    }

    \vspace{7mm}

\begin{frontmatter}

    \title{Refined normal approximations for the central and noncentral chi-square distributions and some applications}%

    \author[a1,a2]{Fr\'ed\'eric Ouimet\texorpdfstring{\corref{cor1}\fnref{fn1}}{)}}%

    \address[a1]{California Institute of Technology, Pasadena, CA 91125, USA.}%
    \address[a2]{McGill University, Montreal, QC H3A 0B9, Canada.}%

    \cortext[cor1]{Corresponding author}%
    \ead{frederic.ouimet2@mcgill.ca}%


    \begin{abstract}
        In this paper, we prove a local limit theorem for the chi-square distribution with $r > 0$ degrees of freedom and noncentrality parameter $\lambda \geq 0$.
        We use it to develop refined normal approximations for the survival function.
        Our maximal errors go down to an order of $r^{-2}$, which is significantly smaller than the maximal error bounds of order $r^{-1/2}$ recently found by \cite{doi:10.1109/LCOMM.2013.111113.131879} and \cite{doi:10.1109/LCOMM.2015.2461681}.
        Our results allow us to drastically reduce the number of observations required to obtain negligible errors in the energy detection problem, from $250$, as recommended in the seminal work of \cite{doi:10.1109/PROC.1967.5573}, to only $8$ here with our new approximations.
        We also obtain an upper bound on several probability metrics between the central and noncentral chi-square distributions and the standard normal distribution, and we obtain an approximation for the median that improves the lower bound previously obtained by \cite{MR1072495}.
    \end{abstract}


\end{frontmatter}

\section{Introduction}\label{sec:intro}

    For any $r > 0$ and $\lambda\geq 0$, the density function of the central and noncentral chi-square distribution $\chi_r^2(\lambda)$ is defined by
    \begin{equation}\label{eq:chi.square.density}
        \begin{aligned}
            f_{r,\lambda}(x)
            &=
                \begin{cases}
                    \frac{(x/2)^{r/2-1}}{2 e^{(x + \lambda)/2}}, &\mbox{if } \lambda = 0 ~~(\text{central}), \\
                    \frac{1}{2} e^{-(x + \lambda)/2} \left(\frac{x}{\lambda}\right)^{(r/2 - 1)/2} I_{r/2 - 1}(\sqrt{\lambda x}), &\mbox{if } \lambda > 0 ~~(\text{noncentral}),
                \end{cases} \\
            &= \frac{(x/2)^{r/2-1}}{2 e^{(x + \lambda)/2}} \sum_{j=0}^{\infty} \frac{(\lambda x / 4)^j}{j! \, \Gamma(r/2 + j)}, \quad x > 0,
        \end{aligned}
    \end{equation}
    using the convention $0^0 \leqdef 1$, and where $I_{\nu}$ denotes the modified Bessel function of the first kind of order $\nu$.
    The special case $\lambda = 0$ corresponds to the central chi-square distribution.
    Sometimes we refer to the chi-square distribution to include the central and noncentral cases all at once.
    When $r = n\in \N$, the expression in \eqref{eq:chi.square.density} corresponds to the density function of $X = \sum_{i=1}^n Z_i^2$, where $Z_i\sim \mathcal{N}(\mu_i, 1)$, and the noncentrality parameter $\lambda$ satisfies $\lambda = \sum_{i=1}^n \mu_i^2$.
    For all $r > 0$, the mean and variance of $X\sim \chi_r^2(\lambda)$ are well known to be
    \begin{equation}\label{eq:noncentral.chi.square.mean.variance}
        \EE[X] = r + \lambda \qquad \text{and} \qquad \VV(X) = 2 (r + 2\lambda),
    \end{equation}
    see, e.g., \cite[Chapter~29, Section~4]{MR1326603}.

    The first goal of our paper (Lemma~\ref{lem:LLT.chi.square}) is to establish a local asymptotic expansion for the ratio of the central and noncentral chi-square density \eqref{eq:chi.square.density} to the normal density with the same mean and variance, namely:
    \begin{equation}\label{eq:phi.M}
        \frac{1}{\sqrt{2 (r + 2\lambda)}} \phi(\delta_x), \quad \text{where} ~~\phi(z) \leqdef \frac{e^{-z^2/2}}{\sqrt{2\pi}} ~~ \text{and} ~~ \delta_x \leqdef \frac{x - (r + \lambda)}{\sqrt{2 (r + 2\lambda)}}.
    \end{equation}

    In \cite{doi:10.1109/LCOMM.2013.111113.131879} and \cite{doi:10.1109/LCOMM.2015.2461681}, the authors derived the following two uniform bounds on a basic approximation of the survival function of the $\chi_r^2(\lambda)$ distribution using the survival function of the standard normal distribution, respectively,
    \begin{align}
        &\max_{a\in \R} \Big|\int_a^{\infty} \hspace{-0.6mm} f_{r,\lambda}(x) \rd x - \int_a^{\infty} \frac{1}{\sqrt{2 (r + 2\lambda)}} \phi(\delta_x) \rd x\Big| \leq \frac{1}{\sqrt{9 \pi r}} + \frac{\widetilde{C}_0}{r}, \quad \text{as } r\to \infty, \label{eq:most.relevant.1} \\
        &\max_{a\in \R} \Big|\int_a^{\infty} \hspace{-0.6mm} f_{r,\lambda}(x) \rd x - \int_a^{\infty} \frac{1}{\sqrt{2 (r + 2\lambda)}} \phi(\delta_x) \rd x\Big| \notag \\
        &\qquad\lessapprox \frac{(r + 4 \lambda)}{\pi W(1) (r + 2 \lambda)^2} \left(1 + \frac{r^2}{32} \binom{r}{8}^{-1/4}\right) + \frac{(r + 3 \lambda)}{2 \sqrt{\pi} (r + 2 \lambda)^{3/2}}, \quad \text{for } r\geq 8, \label{eq:most.relevant.2}
    \end{align}
    where $\widetilde{C}_0 > 0$ is a universal constant, and $W(\cdot)$ is the Lambert $W$-function and $W(1) \simeq 0.5671433$.
    A much older reference, \cite{MR125669}, derived a similar approximation with a maximal error of order $r^{-1/2}$ in the tails. The second goal in our paper is to refine those approximations significantly down to a maximal error of order $r^{-2}$. As a corollary, we can obtain an expansion for the percentage points (or quantiles) of the central and noncentral chi-square distributions in terms of the percentage points of the standard normal distribution. We will do so for the median in Section~\ref{sec:asymptotics.median}, where we improve a previous approximation given by \cite{MR1072495}.

    Here is a brief outline of the paper.
    In Section~\ref{sec:related.works}, we survey the literature on approximations of the cumulative distribution function (henceforth abbreviated by the acronym c.d.f.) and percentage points of the central and noncentral chi-square distribution.
    In Section~\ref{sec:main.results}, we present our main results, which include a local limit theorem for the central and noncentral chi-square distribution and corresponding approximations for the survival function (which improves the results in \cite{doi:10.1109/LCOMM.2013.111113.131879} and \cite{doi:10.1109/LCOMM.2015.2461681}).
    In Section~\ref{sec:applications}, we present two applications of our main results: distance measure bounds between the chi-square distribution and the standard normal distribution, and the asymptotics of the median of the chi-square distribution.
    As mentioned above, the latter improves some results by \cite{MR1072495}.
    The proofs of the main results are gathered in Appendix~\ref{sec:proofs}, and some technical moment calculations are gathered in Appendix~\ref{sec:moments}.

    \begin{notation}
        Throughout the paper, $u = \OO(v)$ means that $\limsup_{r\to \infty} |u / v| < C$, where $C > 0$ is a universal constant.
        Whenever $C$ might depend on some parameter, we add a subscript (for example, $u = \OO_{\lambda}(v)$).
        Similarly, $u = \oo(v)$ means that $\lim |u / v| = 0$ as $r\to \infty$, and subscripts indicate which parameters the convergence rate can depend on.
    \end{notation}

\section{Related works}\label{sec:related.works}

    In addition to the papers of \cite{MR125669}, \cite{doi:10.1109/LCOMM.2013.111113.131879} and \cite{doi:10.1109/LCOMM.2015.2461681}, several other works have discussed normal approximations to the central and/or noncentral chi-square distributions. We briefly mention some of them below.
    For a general reference, some of these approximations are surveyed in \cite[Chapter~29, Section~8]{MR1326603}.
    For the remainder of this section, $\Xi_n$ and $\Xi_{n,\lambda}$ (with $n\in \N$) will denote random variables distributed according to $\chi_n^2(0)$ and $\chi_n^2(\lambda)$, respectively.

    \vspace{3mm}
    \cite{Fisher_1928} shows the approximate normality of a properly translated square root of a central chi-squared random variable, i.e.,
    \begin{equation}
        \sqrt{\Xi_n} - \sqrt{n - 1} ~~\text{is close in law to } \mathcal{N}(0,1), \quad \text{as } n\to \infty.
    \end{equation}
    In a similar fashion, \cite{doi:10.1073/pnas.17.12.684} show the approximate normality of properly translated third roots of central chi-squared random variables, i.e.,
    \begin{align}
        &\sqrt[3]{\Xi_n} - \sqrt[3]{n - 2/3} ~~\text{is close in law to } \mathcal{N}\left(0, \frac{2}{9} \sqrt[3]{\frac{1}{n - 2/3}}\right), \quad \text{as } n\to \infty, \\
        &\sqrt[3]{\Xi_n / n} - \left(1 - \frac{2}{9 n}\right) ~~\text{is close in law to } \mathcal{N}\left(0, \frac{2}{9 n}\right), \quad \text{as } n\to \infty.
    \end{align}
    \cite{MR5588} compares numerically the percentage point approximations derived from the square root transformation of \cite{Fisher_1928} and the third root transformation of \cite{doi:10.1073/pnas.17.12.684}, and he concludes that the latter is significantly more accurate.
    \cite{Germond_Hastings_1944} develop various approximations for the c.d.f.\ of the noncentral chi-square distribution with two degrees of freedom, see \cite[p.466]{MR1326603}.
    \cite{MR15734} uses a method of ``probits'' and ``logits'' to approximate the chi-square distribution, showing the better approximation obtained by the logits.
    \cite{MR34564} gives the following representation of the c.d.f.\ of the noncentral chi-square distribution:
    \begin{equation}
        \PP(\Xi_{n,\lambda} \leq a) \leq \int_0^a \frac{e^{-(x + \lambda)/2}}{2^{n/2}} \sum_{j=0}^{\infty} \frac{x^{n/2 + j - 1} \lambda^j}{\Gamma(n/2 + j) 2^{2j} j!} \rd x.
    \end{equation}
    Patnaik presents many approximations for the c.d.f.\ One line of investigation suggests to approximate the above c.d.f.\ in terms of the central chi-square c.d.f.\ and to combine it with the approximation result of Fisher, which then enables a comparison between the noncentral chi-square c.d.f.\ and the standard normal c.d.f.
    \cite{doi:10.2307/2332731} obtains approximate formulas for the percentage points and the c.d.f.\ of the noncentral chi-square distribution, using the first five cumulants of $(X / (r + \lambda))^{1/3}$, where $X\sim \chi_r^2(\lambda)$, expressed as series in inverse powers of $(r + \lambda)$ up to $(r + \lambda)^{-4}$.
    \cite{MR101581} modifies Abdel-Aty's method by taking $(X / (r + \lambda))^h$, for some $h$ that depends on $r$ and $\lambda$ and which makes the leading third cumulant of $(X / (r + \lambda))^h$ vanish. The objective was to make the distribution of $(X / (r + \lambda))^h$ more nearly normal than that of $(X / (r + \lambda))^{1/3}$. Results are given for a `first (normal) approximation' and a `second approximation', based on a Cornish-Fisher expansion.
    \cite{doi:10.1093/biomet/44.3-4.528} presents an approximation for the 95\% quantile of the noncentral chi-square distribution.
    \cite{MR109379} compares the c.d.f.\ approximations and corresponding percentage point (or quantile) approximations of \cite{MR34564}, \cite{MR109380} and others.

    \vspace{3mm}
    \cite{MR119270} gives an improved Wilson-Hilferty normalized deviate approximation to the chi-square distribution.
    \cite{MR156397} examines a translated version of the third root transformation of a chi-squared random variable, namely $\sqrt[3]{(X - b)/(r + \lambda)}$ for $X\sim \chi_r^2(\lambda)$, and shows how the translation parameter $b$ can be chosen for the approximation to be as good as the `closer approximation' of \cite{doi:10.2307/2332731} for values of the noncentrality parameter that are not too small.
    \cite{MR172370} give an approximation to the c.d.f.\ of the noncentral chi-square distribution in terms of central chi-square distributions, derived from a Laguerre series expansion of the density function. Their approximation add two corrective terms to the one in \cite{MR34564}, see also \cite{MR216616}.
    \cite{MR238470} uses the following expression to approximate the improper integral representation of the survival function of the central chi-square distribution:
    \begin{equation}
        \begin{aligned}
            &\int_a^{\infty} f_{r,0}(x) \rd x \\[-1mm]
            &\approx \frac{e^{-(a - r)/2}}{(a/2 - r/2 + 1) \sqrt{2\pi}} \left(\frac{a}{r}\right)^{r/2} \left(1 - \frac{r/2 - 1}{(a/2 - r/2 + 1)^2 + a}\right) \left(\frac{12 (r/2)^{3/2}}{6 r + 1}\right).
        \end{aligned}
    \end{equation}
    \cite{doi:10.1002/j.1538-7305.1969.tb01111.x} gives the following formula
    \begin{equation}
        \begin{aligned}
            &\int_a^{\infty} \hspace{-0.6mm} f_{r,\lambda}(x) \rd x \\[-1mm]
            &\approx \frac{e^{-\lambda/2}}{\Gamma(r/2)} \cdot \left[\hspace{-1mm}
                \begin{array}{l}
                    \Gamma(r/2,a/2) + \frac{(\lambda/2)}{1!} \left\{\Gamma(r/2,a/2) + (a/r) (a/2)^{r/2 - 1} e^{-a/2}\right\} \\[1mm]
                    + \frac{(\lambda/2)^2}{2!} \left\{\Gamma(r/2,a/2) + \left(\frac{a}{r} + \frac{(a/2)^2}{r/2 (r/2 + 1)}\right) (a/2)^{r/2 - 1} e^{-a/2}\right\} \\[2mm]
                    + \frac{(\lambda/2)^3}{3!} \left\{\hspace{-1mm}
                        \begin{array}{l}
                            \Gamma(r/2,a/2) \\
                            + \left(\frac{a}{r} + \frac{(a/2)^2}{r/2(r/2 + 1)} + \frac{(a/2)^3}{r/2 (r/2 + 1) (r/2 + 2)}\right) (a/2)^{r/2 - 1} e^{-a/2}
                        \end{array}
                        \hspace{-1mm}\right\} \\
                    + \dots
                \end{array}
                \hspace{-1mm}\right]
        \end{aligned}
    \end{equation}
    for ``accurate'' approximations of the c.d.f.\ of the noncentral chi-square distribution over a wide range of degrees of freedom (even over 10,000).

    \vspace{3mm}
    In the same vein as Fisher and Wilson \& Hilferty, \cite{MR595404} and \cite{doi:10.2307/2684608} show the approximate normality of a properly translated fourth root of a chi-squared random variable.
    \cite{MR759243} derive many new integral representations for the c.d.f.\ of the noncentral chi-square distribution; the numerical usefulness remains unclear.
    \cite{MR980908} approximate the density and c.d.f.\ of a noncentral chi-square distribution using a table of modified Bessel functions.
    They give an exact expression when the degrees of freedom are odd.
    \cite{MR1012493} presents two formulas to approximate the c.d.f.\ of the noncentral chi-square distribution, namely the following first and second order Wiener germ approximations:
    \begin{align}
        &\PP(\Xi_{r,\lambda} \leq a) \approx \Phi\left(\pm \sqrt{
            \begin{array}{l}
                r (s - 1)^2 \left(\frac{1}{2 s} + \mu^2 - \frac{1}{s} h(1 - s)\right) \\
                - \log\left(\frac{1}{s} - \frac{2}{s} \cdot \frac{h(1 - s)}{(1 + 2 \mu^2 s)}\right) + \frac{2 (1 + 3 \mu^2)^2}{9 r (1 + 2 \mu^2)^3}
            \end{array}
            }\right), \\
        &\PP(\Xi_{r,\lambda} \leq a) \approx \Phi\left(\pm \sqrt{
            \begin{array}{l}
                r (s - 1)^2 \left(\frac{1}{2 s} + \mu^2 - \frac{1}{s} h(1 - s)\right) \\
                - \log\left(\frac{1}{s} - \frac{2}{s} \cdot \frac{h(1 - s)}{(1 + 2 \mu^2 s)}\right) + \frac{2}{r} B(s)
            \end{array}
            }\right),
    \end{align}
    where
    \begin{equation}
        \begin{aligned}
            B(s)
            &= - \frac{3}{2} \cdot \frac{(1 + 4 \mu^2 s)}{(1 + 2 \mu^2 s)^2} + \frac{5}{3} \cdot \frac{(1 + 3 \mu^2 s)^2}{(1 + 2 \mu^2 s)^3} + \frac{2 (1 + 3 \mu^2 s)}{(s - 1) (1 + 2 \mu^2 s)^2} \\
            &\quad+ \frac{3 \eta}{(s - 1)^2 (1 + 2 \mu^2 s)} - \frac{(1 + 2 h(\eta)) \eta^2}{2 (s - 1)^2 (1 + 2 \mu^2 s)}.
        \end{aligned}
    \end{equation}
    and
    \begin{equation}
        \begin{aligned}
            &h(y) =
            \begin{cases}
                \frac{1}{y^2} \left[(1 - y) \log(1 - y) + y - \frac{y^2}{2}\right], &\mbox{if } h\in (0,1), \\
                0, &\mbox{if } h = 0,
            \end{cases} \\
            &\mu = \frac{\lambda}{r}, \quad s = \frac{-1 + \sqrt{1 + (4 x \mu^2) / r}}{2 \mu^2} ~~\text{for some } x > 0, \\[1mm]
            &\eta = \frac{1 + 2 \mu^2 s - 2 h(1 - s) - s - 2 \mu^2 s^2}{1 + 2 \mu^2 s - 2 h(1 - s)}.
        \end{aligned}
    \end{equation}
    The numerical implementation  of these formulas is given by \cite{doi:10.1007/s001800000029}.
    \cite{MR1097704} give two computational approximations to the c.d.f.\ of the noncentral chi-square distribution of any degree of freedom and odd degrees of freedom respectively, using truncated infinite sums:
    \begin{align}
        &\PP(\Xi_{n,\lambda} \leq a) = \frac{e^{-(a + \lambda)/2} (a/2)^{n/2}}{\Gamma(n/2 + 1)} \sum_{i=0}^{\infty} \frac{1}{i!} C_i(\lambda a/4,n/2) \sum_{s=0}^{\infty} C_s(a/2,n/2+i), \\
        &\qquad\text{where } C_i(\lambda a/4,n/2) = \frac{\lambda a/4}{n/2 + i} C_{i-1}(\lambda a/4,n/2), \quad \text{for } i = 1,2,3,\dots \notag \\
        &\qquad\text{and } C_0(\lambda a/4,n/2) = 1, \notag \\
        &\PP(\Xi_{2n+1,\lambda} > a) = \PP(\Xi_{2n+1} > a) + \sqrt{\frac{2}{\pi}} e^{-a/2} \sum_{k=n+1}^{\infty} \frac{a^{k-1}}{(2k-1)!!} \PP(\Xi_{2k - 2n} \leq \lambda).
    \end{align}
    \cite{MR1072495} gives bounds on the quantiles of the noncentral chi-square distribution in terms of the noncentrality parameter. Increasing the accuracy requires shortening the range of the noncentrality parameter.
    \cite{doi:10.2307/2347584} gives an algorithm to compute the noncentral chi-square c.d.f.\ using a series representation based on a Poisson weighted sum of central chi-square c.d.f.s. \cite{MR1199912} gives two asymptotic expansions for the survival function of the noncentral chi-square distribution involving the survival function of the standard normal distribution, namely
    \begin{equation}
        \begin{aligned}
            &\int_a^{\infty} \hspace{-0.6mm} f_{r,\lambda}(x) \rd x \sim 1 - \left(\frac{a}{\lambda}\right)^{r/4-1/4} \Psi(\sqrt{r/2} - \sqrt{\lambda/2}), \\
            &\int_a^{\infty} \hspace{-0.6mm} f_{r,\lambda}(x) \rd x \sim \Psi\left(-u_0 \sqrt{\frac{\lambda a}{2}}\right) + \phi\left(\frac{u_0}{\sqrt{\lambda a}}\right) \sum_{k=0}^{\infty} c_{2k} \frac{\Gamma(k + 1/2)}{\Gamma(1/2)} \left(\frac{2}{\lambda a}\right)^k, \quad \lambda a\to \infty,
        \end{aligned}
    \end{equation}
    for specific constants $u_0,c_0,c_2,\dots$ given in \cite[p.60]{MR1199912}.
    \cite{doi:10.1080/03610919508813266} obtain an alternative error bound on Ruben's algorithm \citep{MR418308} for the computation of the noncentral chi-square c.d.f. They also compare finite algorithms (such as \cite{MR34564} and others) with the algorithms proposed by \cite{MR1097704} for such computation and they discuss the rates of convergence of two different series representations for the c.d.f.
    \cite{MR1625949} uses third order asymptotic methods that only requires evaluation of the standard normal to approximate the c.d.f.\ of the noncentral chi-square distribution.

    \vspace{3mm}
    \cite{MR2133578} approximates the c.d.f.\ of a central chi-square distribution by considering a linear combination of fractional powers of a chi-squared random variable. The mean absolute error is shown to be lower than other power transformations (two of the most well known are the square root transformation by \cite{doi:10.2307/2340521} and the third root transformation by \cite{doi:10.1073/pnas.17.12.684}) for degrees of freedom $1 \leq r \leq 1000$.
    \cite{MR3655852} uses Stein's method to obtain an order $n^{-1}$ bound on the distributional distance over smooth test functions between Pearson's statistics and its limiting chi-square distribution.
    \cite{MR3619428} gives three formulas for the c.d.f.\ of the noncentral chi-square distribution in terms of modified Bessel functions, leaky aquifer functions, and generalized incomplete gamma functions, respectively.
    \cite{Okagbue_Adamu_Anake_2017} uses quantile mechanics methods to approximate the quantile density function (the derivative of the quantile function) and the corresponding quartiles of the chi-square distribution. The result of the method is a power series solution to an ordinary differential equation.
    \cite{doi:10.3390/math9020129} give various representations of the noncentral chi-square c.d.f.\ in terms of modified Bessel functions of the first kind, derived from two mean value theorems for definite integrals.
    \cite{arXiv:2111.00949} uses Stein's method to obtain an order $n^{-1}$ bound on the distributional distance over smooth test functions between Friedman's statistics and its limiting chi-square distribution.

    \vspace{3mm}
    For a discussion on the estimation of quadratic forms or the noncentrality parameter for a noncentral chi-square distribution, we refer the reader to \cite{MR331611}, \cite{MR378217}, \cite{MR0397966}, \cite{MR661020}, \cite{MR663453}, \cite{MR832996}, \cite{MR888440}, \cite{MR1221856}, \cite{MR1333127}, \cite{MR1836572,MR1863157}, \cite{MR2657050} and \cite{MR3667406}.

\section{Main results}\label{sec:main.results}

    First, we need local approximations for the ratio of the noncentral chi-square density to the normal density function with the same mean and variance.

    \begin{lemma}[Local approximation]\label{lem:LLT.chi.square}
        For any $r > 0$, $0 \leq \lambda = \oo(\sqrt{r})$ and $\eta\in (0,1)$, define
        \begin{equation}
            D_{r,\lambda} \leqdef \frac{r + \lambda}{\sqrt{2 (r + 2\lambda)}},
        \end{equation}
        and let
        \begin{equation}\label{eq:bulk}
            B_{r,\lambda}(\eta) \leqdef \bigg\{x\in (0,\infty) : \bigg|\frac{\delta_x}{D_{r,\lambda}}\bigg| \leq \eta \, r^{-1/3}\bigg\},
        \end{equation}
        denote the bulk of the noncentral chi-square distribution.
        Then, uniformly for $k\in B_{r,\lambda}(\eta)$, we have, as $r\to \infty$,
        \begin{align}\label{eq:lem:LLT.chi.square.eq.log}
            \log\bigg(\frac{f_{r,\lambda}(x)}{\frac{1}{\sqrt{2 (r + 2\lambda)}} \phi(\delta_x)}\bigg)
            &= r^{-1/2} \, \bigg\{\frac{\sqrt{2}}{3} \, \delta_x^3 - \sqrt{2} \, \delta_x\bigg\} + r^{-1} \, \bigg\{-\frac{1}{2} \, \delta_x^4 + \delta_x^2 - \frac{1}{6}\bigg\} \notag \\[-2.5mm]
            &\quad+ r^{-3/2} \, \bigg\{\frac{2^{3/2}}{5} \, \delta_x^5 - \frac{2^{3/2}}{3} \, \delta_x^3\bigg\} + \OO\bigg(\frac{(1 \vee \lambda^4) + \lambda^2 |\delta_x|^2}{r^2}\bigg),
        \end{align}
        Furthermore,
        \begin{align}\label{eq:lem:LLT.chi.square.eq}
            \frac{f_{r,\lambda}(x)}{\frac{1}{\sqrt{2 (r + 2\lambda)}} \phi(\delta_x)} = 1
            &+ r^{-1/2} \, \bigg\{\frac{\sqrt{2}}{3} \, \delta_x^3 - \sqrt{2} \, \delta_x\bigg\} + r^{-1} \, \bigg\{\frac{1}{9} \, \delta_x^6 - \frac{7}{6} \, \delta_x^4 + 2 \, \delta_x^2 - \frac{1}{6}\bigg\} \notag \\[-2mm]
            &\quad+ r^{-3/2} \, \bigg\{\frac{\sqrt{2}}{81} \, \delta_x^9 - \frac{5}{9\sqrt{2}} \, \delta_x^7 + \frac{47}{15\sqrt{2}} \, \delta_x^5 - \frac{37}{9\sqrt{2}} \, \delta_x^3 + \frac{1}{3\sqrt{2}} \, \delta_x\bigg\} \notag \\[0.5mm]
            &\quad+ \OO_{\eta}\bigg(\frac{(1 \vee \lambda^4) + |\delta_x|^{10}}{r^2}\bigg).
        \end{align}
    \end{lemma}
    For the interested reader, local approximations akin to Lemma~\ref{lem:LLT.chi.square} were derived for the Poisson, binomial, negative binomial, multinomial, Dirichlet, Wishart and multivariate hypergeometric distributions in \cite[Lemma~2.1]{MR4213687}, \cite[Lemma~3.1]{MR4340237}, \cite[Lemma~2.1]{arXiv:2103.08846}, \cite[Theorem~2.1]{MR4249129}, \cite[Theorem~1]{MR4394974}, \cite[Theorem~1]{MR4358612}, \cite[Theorem~1]{MR4361955}, respectively.
    See also earlier references such as \cite{MR207011} (based on Fourier analysis results from \cite{MR14626}) for the Poisson, binomial and negative binomial distributions, and \cite{MR538319} for the binomial distribution.

    By integrating the second local approximation in Lemma~\ref{lem:LLT.chi.square}, we can approximate the survival function of the $\chi_r^2(\lambda)$ distribution, i.e.,
    \begin{equation}
        S_{r,\lambda}(a) \leqdef \int_a^{\infty} \hspace{-0.6mm} f_{r,\lambda}(x) \rd x, \quad a\in \R,
    \end{equation}
    using the survival function of the normal distribution with the same mean and variance.

    \begin{theorem}[Survival function approximations]\label{thm:refined.approximations}
        For any $0 \leq \lambda = \oo(\sqrt{r})$, we have, as $r\to \infty$,
        \begin{align}
            &\hspace{-2cm}\text{Order 0 approximation:} \notag \\[-0.3mm]
            &E_0 \leqdef \max_{a\in \R} \left|S_{r,\lambda}(a) - \Psi(\delta_a)\right| \leq \frac{M_0}{r^{1/2}} + \frac{C_0}{r}, \label{eq:order.0.approx} \\[0.5mm]
            &\hspace{-2cm}\text{Order 1 approximation:} \notag \\[-0.3mm]
            &E_1 \leqdef \max_{a\in \R} \left|S_{r,\lambda}(a) - \Psi\big(\delta_{a - d_1}\big)\right| \leq \frac{M_1}{r} + \frac{C_1 \, (1 \vee \lambda)}{r^{3/2}}, \label{eq:order.1.approx} \\[0.5mm]
            &\hspace{-2cm}\text{Order 2 approximation:} \notag \\[-0.3mm]
            &E_2 \leqdef \max_{a\in \R} \left|S_{r,\lambda}(a) - \Psi\big(\delta_{a - (d_1 + \frac{d_2}{\sqrt{r}})}\big)\right| \leq \frac{M_2}{r^{3/2}} + \frac{C_2 \, (1 \vee \lambda^4)}{r^2}, \label{eq:order.2.approx} \\[0.5mm]
            &\hspace{-2cm}\text{Order 3 approximation:} \notag \\[-0.3mm]
            &E_3 \leqdef \max_{a\in \R} \left|S_{r,\lambda}(a) - \Psi\big(\delta_{a - (d_1 + \frac{d_2}{\sqrt{r}} + \frac{d_3}{r})}\big)\right| \leq \frac{C_3 \, (1 \vee \lambda^4)}{r^2}, \label{eq:order.3.approx}
        \end{align}
        where $\Psi$ denotes the survival function of the standard normal distribution, $C_i,~0 \leq i \leq 3,$ are universal constants, and
        \begin{equation}\label{eq:constants.d1.d2.d3}
            \begin{aligned}
                d_1 &\leqdef \frac{2}{3} (\delta_a^2 - 1), \qquad d_2 \leqdef \frac{1}{9\sqrt{2}} (\delta_a - 7 \delta_a^3), \\[0.5mm]
                d_3 &\leqdef \frac{1}{405} \left(219 \delta_a^4 + (270 \lambda - 14) \delta_a^2 - (270 \lambda + 13)\right), \\
                M_0 &\leqdef \max_{y\in \R} \frac{\sqrt{2}}{3} |y^2 - 1| \phi(y) = \frac{1}{\sqrt{9 \pi}} = 0.188063\dots, \\
                M_1 &\leqdef \max_{y\in \R} \frac{1}{18} |7 y^3 - y| \phi(y) = 0.171448\dots, \\
                M_2 &\leqdef \max_{y\in \R} \frac{1}{405\sqrt{2}} \left|219 y^4 + (270 \lambda - 14) y^2 - (270 \lambda + 13)\right| \phi(y), \\
                &\qquad (\text{$M_2$ is equal to $0.326258$ when $\lambda = 0$}).
            \end{aligned}
        \end{equation}
        The constants $M_0$, $M_1$, $M_2$ are illustrated numerically in Figure~\ref{fig:asymptotics.E0.E1.E2}, for multiple values of $\lambda$.
        The maximal errors are plotted as a function of $r$ in Figure~\ref{fig:loglog.errors.plot}.
    \end{theorem}

    In \cite{doi:10.1109/LCOMM.2015.2461681}, it is mentioned that when using the $r = 250$ recommendation of \cite{doi:10.1109/PROC.1967.5573} for the energy detection problem, the maximal error is $0.01516183$ in \eqref{eq:most.relevant.2} for $\lambda = 0$ (the central chi-square approximation).
    When using the Order~$1$~and~$2$ approximations in Theorem~\ref{thm:refined.approximations} and ignoring the error terms of order $r^{-3/2}$ and $r^{-2}$, we would only need $r = 12$ and $r = 8$, respectively, to achieve a smaller maximal error, which is a significant improvement (although we have to keep in mind that the error bounds are asymptotic).

    \vspace{3mm}
    \begin{figure}[!ht]
        \captionsetup{width=0.7\linewidth}
        \centering
        \begin{subfigure}[b]{0.32\textwidth}
            \centering
            \includegraphics[width=\textwidth, height=0.85\textwidth]{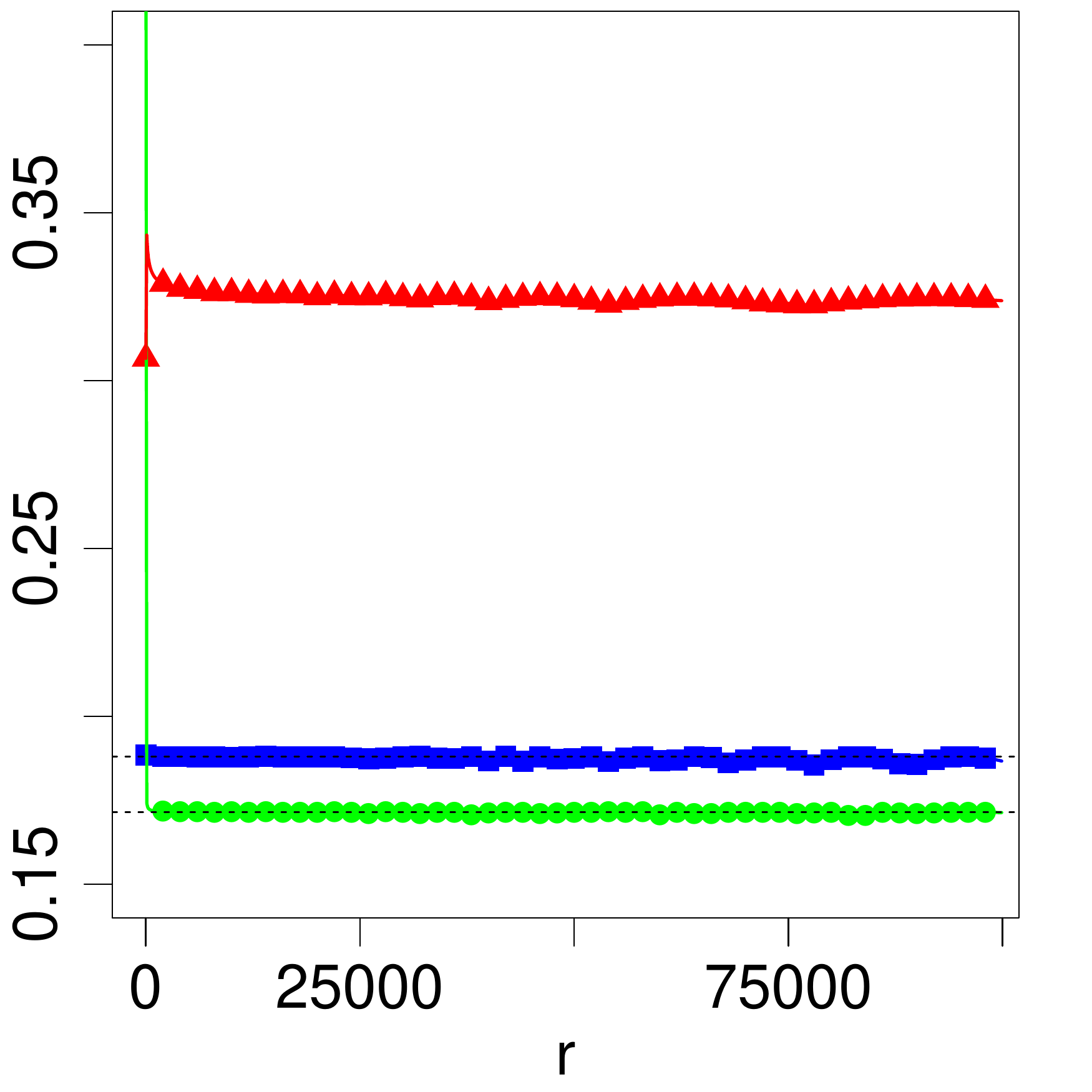}
            \vspace{-6mm}
            \caption{$\lambda = 0$}
        \end{subfigure}
        \begin{subfigure}[b]{0.32\textwidth}
            \centering
            \includegraphics[width=\textwidth, height=0.85\textwidth]{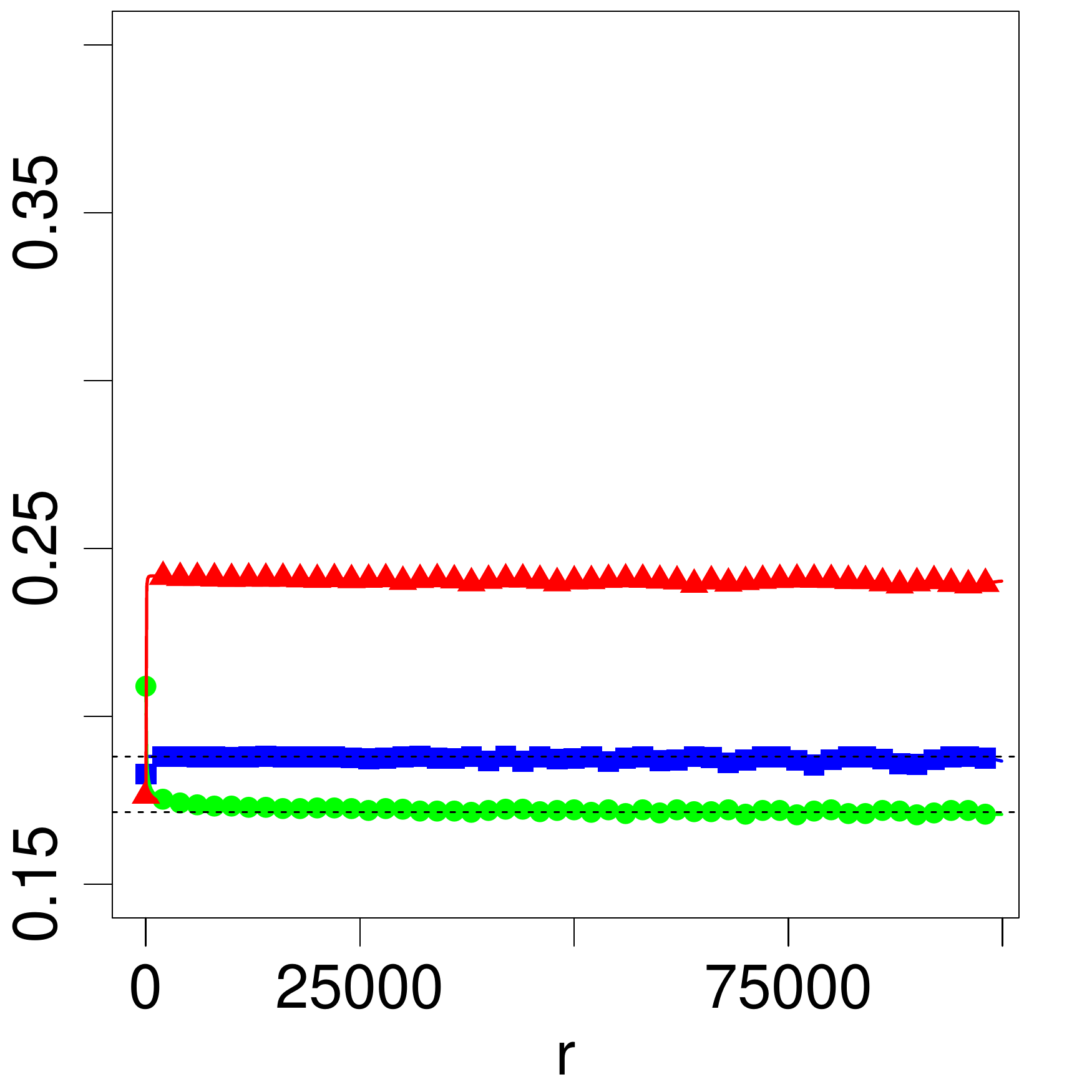}
            \vspace{-6mm}
            \caption{$\lambda = 2$}
        \end{subfigure}
        \begin{subfigure}[b]{0.32\textwidth}
            \centering
            \includegraphics[width=\textwidth, height=0.85\textwidth]{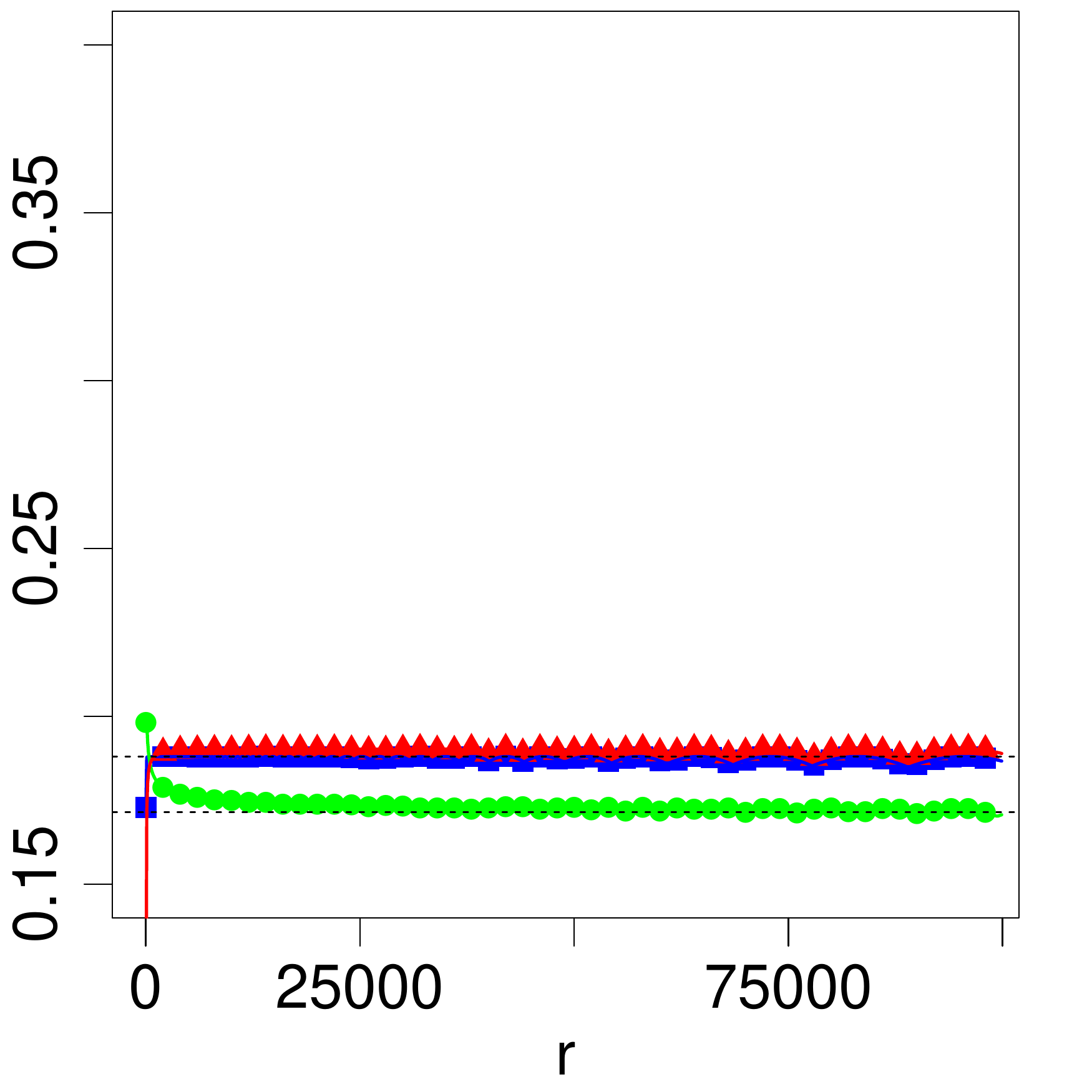}
            \vspace{-6mm}
            \caption{$\lambda = 4$}
        \end{subfigure}
        \begin{subfigure}[b]{0.32\textwidth}
            \vspace{3mm}
            \centering
            \includegraphics[width=\textwidth, height=0.85\textwidth]{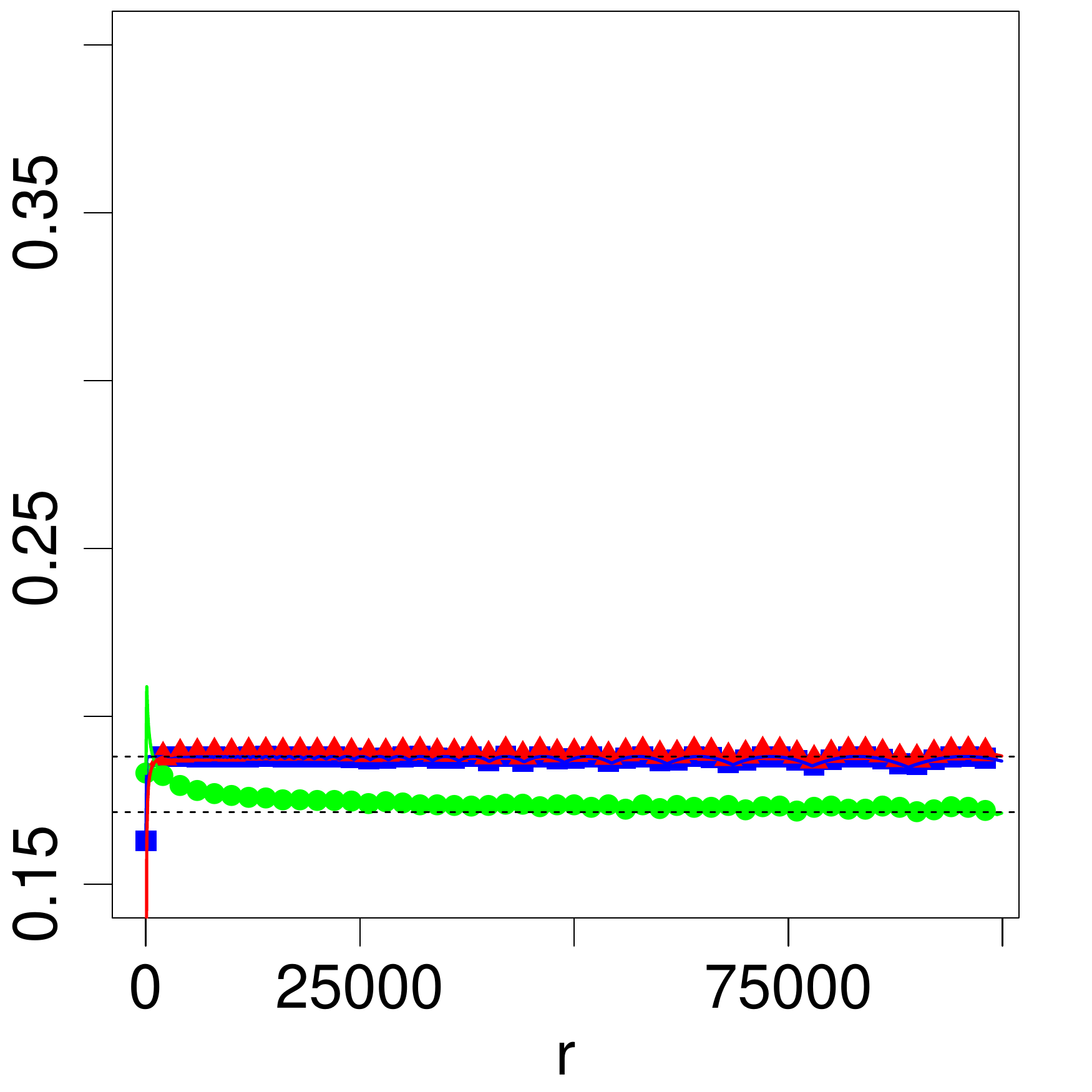}
            \vspace{-6mm}
            \caption{$\lambda = 6$}
            \end{subfigure}
        \begin{subfigure}[b]{0.32\textwidth}
            \centering
            \includegraphics[width=\textwidth, height=0.85\textwidth]{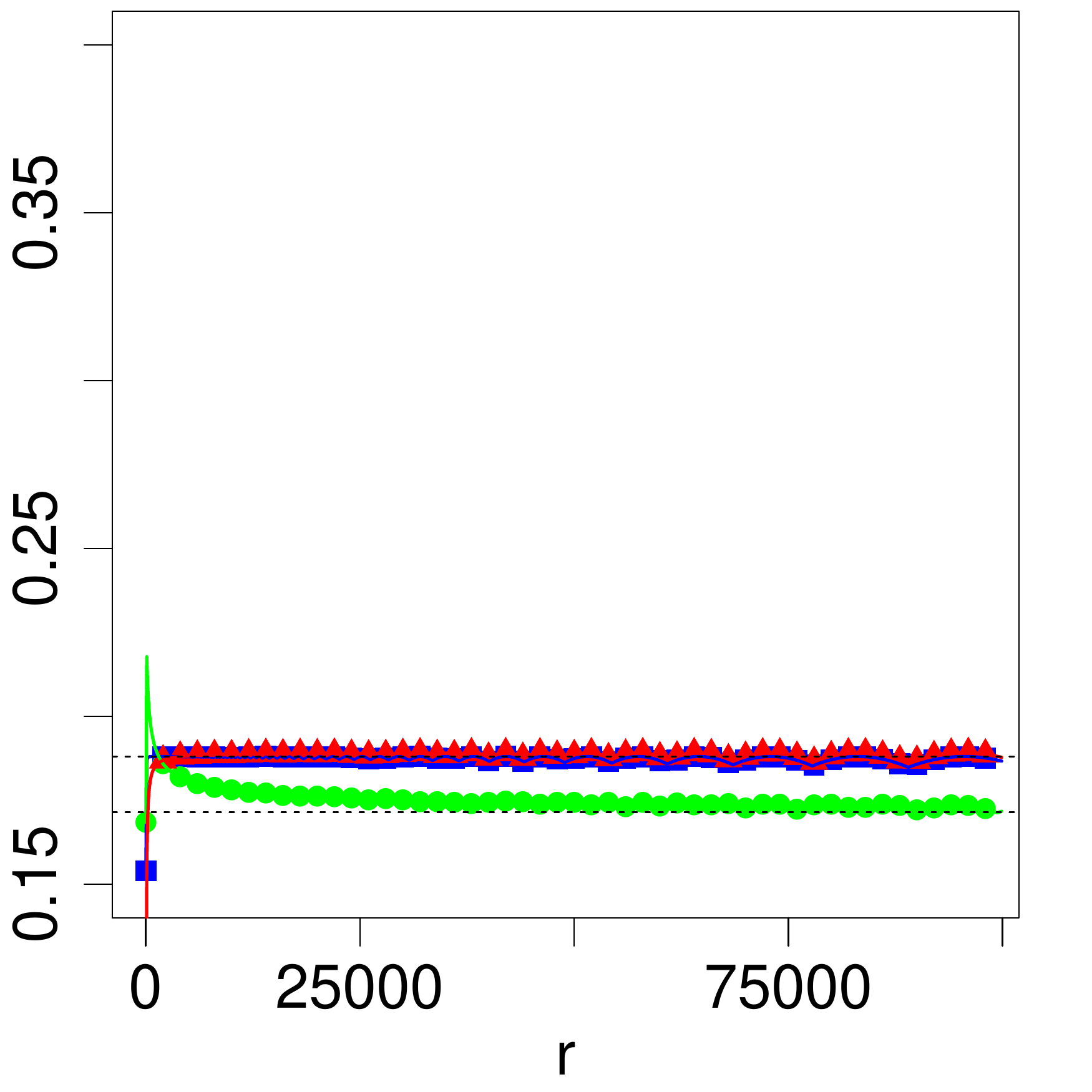}
            \vspace{-6mm}
            \caption{$\lambda = 8$}
        \end{subfigure}
        \begin{subfigure}[b]{0.32\textwidth}
            \centering
            \includegraphics[width=\textwidth, height=0.85\textwidth]{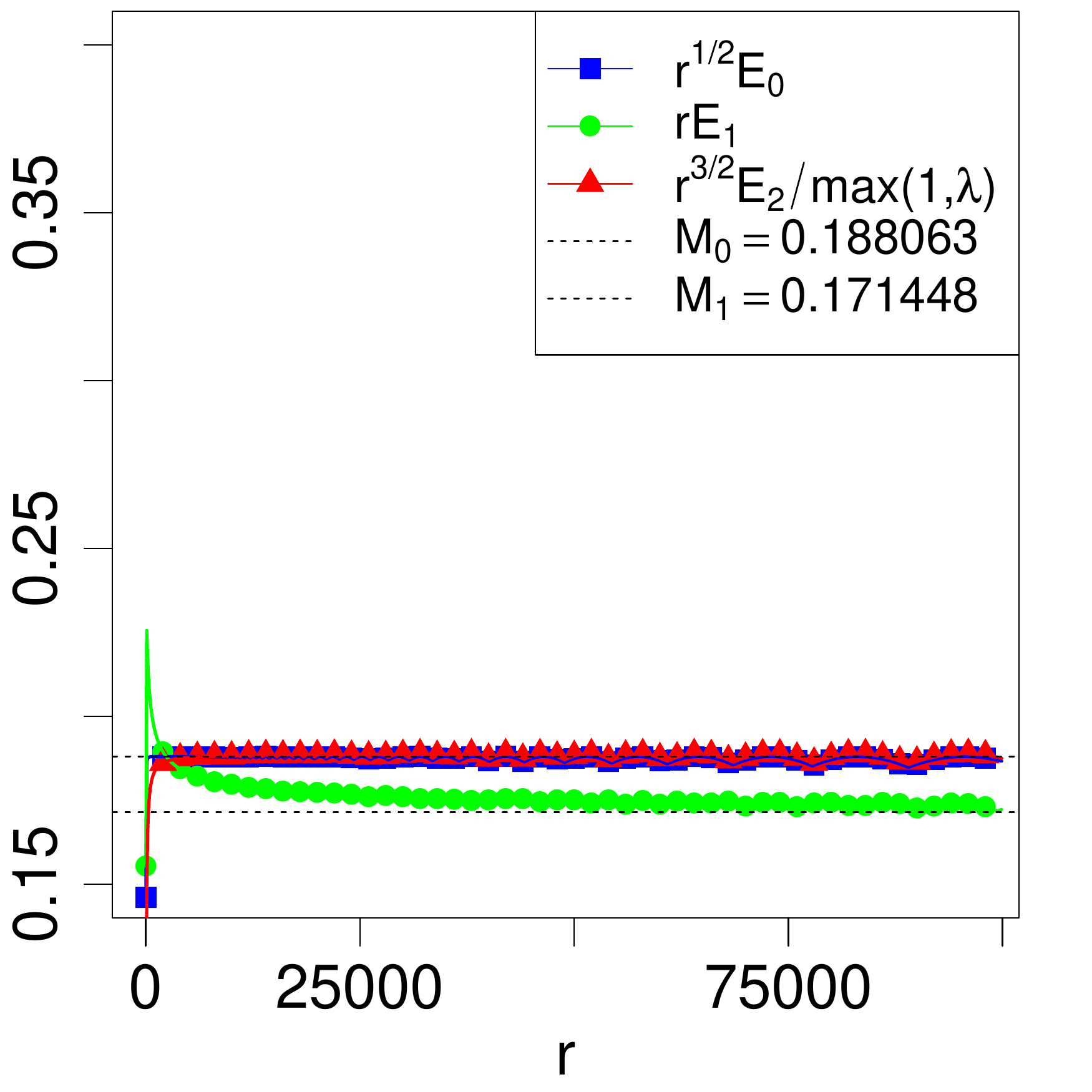}
            \vspace{-6mm}
            \caption{$\lambda = 10$}
        \end{subfigure}
        \caption{Numerical illustration of the asymptotic constants $M_0$, $M_1$ and $M_2$, for each $\lambda\in \{0,2,4,6,8,10\}$.}
        \label{fig:asymptotics.E0.E1.E2}
    \end{figure}

    \begin{figure}[!ht]
        \captionsetup{width=0.7\linewidth}
        \centering
        \includegraphics[width=0.6\textwidth]{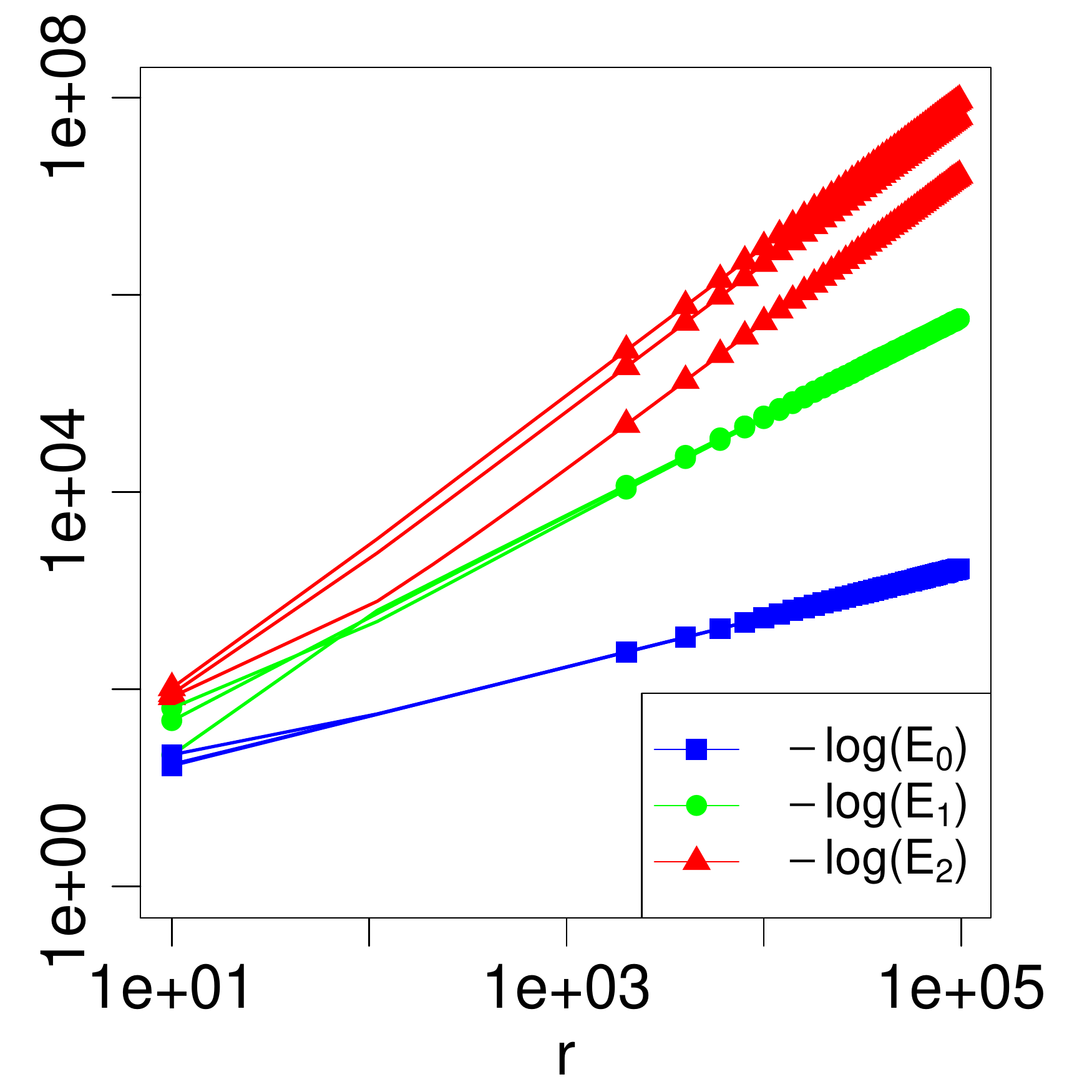}
        \caption{Log-log plot for the maximum absolute errors $E_0$, $E_1$ and $E_2$, as a function of $r$, for $\lambda = 0$, $\lambda = 2$ and $\lambda = 10$.}
        \label{fig:loglog.errors.plot}
    \end{figure}

\newpage
\vspace{5cm}
\section{Applications}\label{sec:applications}

    \subsection{Probability metrics upper bounds between chi-square and normal distributions}\label{sec:total.variation}

    For our first application, we use Lemma~\ref{lem:LLT.chi.square} to compute an upper bound on the total variation between the probability measures induced by \eqref{eq:chi.square.density} and \eqref{eq:phi.M}.
    Given the relation there is between the total variation and other probability metrics such as the Hellinger distance (see, e.g., \citet[p.421]{doi:10.2307/1403865}), we obtain upper bounds on other distance measures automatically.

    \begin{theorem}[Probability metrics bounds]\label{thm:total.variation}
        Let $r > 0$ and $0 \leq \lambda = \oo(\sqrt{r})$.
        Let $\PP_{r,\lambda}$ be the law of the $\chi_r^2(\lambda)$ distribution.
        Let $\QQ_{r,\lambda}$ be the law of the $\mathrm{Normal}(r+\lambda, 2 (r + 2\lambda))$ distribution.
        Then, as $r\to \infty$, we have
        \begin{equation}
            \mathrm{dist}\hspace{0.3mm}(\PP_{r,\lambda},\QQ_{r,\lambda}) \leq \frac{C}{\sqrt{r}} \qquad \text{and} \qquad \mathcal{H}(\PP_{r,\lambda},\QQ_{r,\lambda}) \leq \sqrt{\frac{2 C}{\sqrt{r}}},
        \end{equation}
        where $C > 0$ is a universal constant, $\mathcal{H}(\cdot,\cdot)$ denotes the Hellinger distance, and $\mathrm{dist}\hspace{0.3mm}(\cdot,\cdot)$ can be replaced by any of the following probability metrics: Total variation, Kolmogorov (or Uniform) metric, L\'evy metric, Discrepancy metric, Prokhorov metric.
    \end{theorem}

    \begin{proof}[Proof of Theorem~\ref{thm:total.variation}]
        Let $X\sim \PP_{r,\lambda}$.
        By the comparison of the total variation norm $\|\cdot\|$ with the Hellinger distance on page 726 of \cite{MR1922539}, we already know that
        \begin{equation}\label{eq:first.bound.total.variation}
            \|\PP_{r,\lambda} - \QQ_{r,\lambda}\| \leq \sqrt{2 \, \PP(X\in B_{r,\lambda}^{\hspace{0.2mm}c}(1/2)) + \EE\bigg[\log\left(\frac{\rd \PP_{r,\lambda}}{\rd \QQ_{r,\lambda}}(X)\right) \, \ind_{\{X\in B_{r,\lambda}(1/2)\}}\bigg]}.
        \end{equation}
        Then, by applying a large deviation bound for the noncentral chi-square distribution (for example combine the Order~$0$ approximation in Theorem~\ref{thm:refined.approximations} with a Mills ratio Gaussian tail inequality), we get, for $r$ large enough,
        \begin{align}\label{eq:concentration.bound}
            \PP(X\in B_{r,\lambda}^{\hspace{0.2mm}c}(1/2))
            &\leq 100 \, \exp\left(-\frac{1}{100} r^{1/3}\right).
        \end{align}
        By Lemma~\ref{lem:LLT.chi.square}, we have
        \begin{equation}\label{eq:estimate.I.begin}
            \begin{aligned}
                &\EE\bigg[\log\bigg(\frac{\rd \PP_{r,\lambda}}{\rd \QQ_{r,\lambda}}(X)\bigg) \, \ind_{\{X\in B_{r,\lambda}(1/2)\}}\bigg] \\
                &\quad= r^{-1/2} \cdot \EE\left[\bigg\{\frac{\sqrt{2}}{3} \cdot \frac{(X - (r + \lambda))^3}{(2 (r + 2 \lambda))^{3/2}} - \sqrt{2} \cdot \frac{(X - (r + \lambda))}{(2 (r + 2 \lambda))^{1/2}}\bigg\} \, \ind_{\{X\in B_{r,\lambda}(1/2)\}}\right] \\
                &\quad\quad+ r^{-1} \cdot \OO\left(\frac{\EE[(X - (r + \lambda))^4]}{(2 (r + 2 \lambda))^2} + \frac{\EE[(X - (r + \lambda))^2]}{2 (r + 2 \lambda)} + 1\right) \\
                &\quad\quad+ \OO(r^{-3/2}).
            \end{aligned}
        \end{equation}
        By Lemma~\ref{lem:central.moments.chi.square} and Corollary~\ref{cor:central.moments.chi.square.on.events}, we get
        \begin{equation}\label{eq:estimate.I}
            \begin{aligned}
                &\EE\bigg[\log\bigg(\frac{\rd \PP_{r,\lambda}}{\rd \QQ_{r,\lambda}}(X)\bigg) \, \ind_{\{X\in B_{r,\lambda}(1/2)\}}\bigg] \\[0.5mm]
                &\quad= r^{-1/2} \cdot (\PP(X\in B_{r,\lambda}^{\hspace{0.2mm}c}(1/2)))^{1/2} + r^{-1} \cdot \OO(1) + \OO(r^{-3/2}) \\
                &\quad= \OO(r^{-1}).
            \end{aligned}
        \end{equation}
        This ends the proof.
    \end{proof}

    \subsection{Asymptotics of the median}\label{sec:asymptotics.median}

    For our second application, we improve the lower bound for the median of the noncentral chi-square distribution found in Proposition~4.1 of \cite{MR1072495} (the upper bounds are comparable).
    The proof relies on the refined normal approximation in Theorem~\ref{thm:refined.approximations}, a Taylor expansion for the c.d.f.\ of the standard normal distribution, and solving a quadratic equation involving the normalized (via $\delta_{\cdot}$) median.

    \begin{theorem}\label{thm:asymptotics.median}
        Let $r > 0$ and $0 \leq \lambda = \oo(\sqrt{r})$, and let $X\sim \chi_r^2(\lambda)$.
        Then, we have
        \begin{equation}
            \mathrm{Median}\hspace{0.3mm}(X) = r + \lambda - \frac{2}{3} + \OO\left(\frac{1}{r^{1/2}}\right) + \OO\left(\frac{1 \vee \lambda}{r}\right), \quad \text{as } r\to \infty. \label{eq:thm:asymptotics.median.eq.1}
        \end{equation}
    \end{theorem}

    \begin{proof}[Proof of Theorem~\ref{thm:asymptotics.median}]
        By definition, the median of the $\chi_r^2(\lambda)$ distribution is the point $a^{\star} > 0$ that satisfies $S_{r,\lambda}(a^{\star}) = 1/2$.
        By the Order 1 approximation in Theorem~\ref{thm:refined.approximations}, we want to find $a^{\star}$ such that
        \begin{equation}\label{eq:asymptotics.median.eq.begin}
            \left|\Psi\big(\delta_{a^{\star} - d_1}\big) - \frac{1}{2}\right| \leq \frac{M_1}{r} + \frac{C_1 (1 \vee \lambda)}{r^{3/2}}.
        \end{equation}
        A Taylor expansion for $\Psi$ at $0$ yields
        \begin{equation}
            \Psi(x) = \frac{1}{2} - \frac{x}{\sqrt{2\pi}} + \OO(x^3), \quad \text{as } x\to 0.
        \end{equation}
        Equation~\eqref{eq:asymptotics.median.eq.begin} then becomes
        \begin{equation}\label{eq:asymptotics.median.eq.next}
            \left|a^{\star} - d_1 - (r + \lambda) + \OO(r^{-1})\right| \leq \frac{\widetilde{M}_1}{r^{1/2}} + \frac{\widetilde{C}_1 (1 \vee \lambda)}{r},
        \end{equation}
        for appropriate universal constants $\widetilde{M}_1, \widetilde{C}_1 > 0$.
        The error $\OO(r^{-1})$ in \eqref{eq:asymptotics.median.eq.next} does not depend on $\lambda$ because $\delta_{a^{\star} - d_1} = \OO(r^{-1/2})$ by the a priori bounds we have from Proposition~4.1 in \cite{MR1072495}.
        From \eqref{eq:asymptotics.median.eq.next} and the expression for $d_1$ (at $a = a^{\star}$) in \eqref{eq:constants.d1.d2.d3}, we deduce
        \begin{equation}\label{eq:polynomial.equation}
            \delta_{a^{\star}} = \frac{\frac{2}{3} (\delta_{a^{\star}}^2 - 1)}{\sqrt{2 (r + 2 \lambda)}} + \OO\left(\frac{1}{r}\right) + \OO\left(\frac{1 \vee \lambda}{r^{3/2}}\right).
        \end{equation}
        This quadratic equation in the variable $\delta_{a^{\star}}$ yields the following two solutions (with the notation $\varepsilon_{r,\lambda} \leqdef 1/\sqrt{2 (r + 2 \lambda)}$):
        \begin{equation}\label{eq:asymptotics.median.eq.quadratic}
            (\delta_{a^{\star}})_{1,2} = \frac{-1 \pm \sqrt{1 - 4 \cdot \frac{-2}{3} \varepsilon_{r,\lambda} \cdot \big[\frac{2}{3} \varepsilon_{r,\lambda} - \OO\big(\frac{1}{r}\big) + \OO\big(\frac{1 \vee \lambda}{r^{3/2}}\big)\big]}}{2 \cdot \frac{-2}{3} \varepsilon_{r,\lambda}}
        \end{equation}
        Because of the a priori bounds on the median in Proposition~4.1 of \cite{MR1072495}, the unique solution must be the one with the minus in \eqref{eq:asymptotics.median.eq.quadratic}.
        Therefore, the median $a^{\star}$ satisfies
        \begin{equation}
            a^{\star} - (r + \lambda) = \frac{1 - \sqrt{1 + \frac{16}{9} \varepsilon_{r,\lambda}^2 + \OO\big(\frac{1}{r^{3/2}}\big) + \OO\big(\frac{1 \vee \lambda}{r^2}\big)}}{\frac{4}{3} \varepsilon_{r,\lambda}^2}.
        \end{equation}
        Using the Taylor expansion $\sqrt{1 + y} = 1 + \frac{y}{2} + \OO(y^2), ~\text{as } y\to 0$, we have
        \begin{equation}
            a^{\star} - (r + \lambda) = -\frac{2}{3} + \OO\left(\frac{1}{r^{1/2}}\right) + \OO\left(\frac{1 \vee \lambda}{r}\right).
        \end{equation}
        This ends the proof.
    \end{proof}

    \begin{remark}
        It is possible to use the higher order approximations from Theorem~\ref{thm:refined.approximations} and apply the same logic in the proof of Theorem~\ref{thm:asymptotics.median} to derive an expression for the median which is asymptotically more precise, but the algebra becomes much uglier. In particular, the degree of the polynomial equation to solve in \eqref{eq:polynomial.equation} will increase.
    \end{remark}

\appendix

\begin{appendices}

\section{Proofs}\label{sec:proofs}

    \begin{proof}[Proof of Lemma~\ref{lem:LLT.chi.square}]
        By taking the logarithm in \eqref{eq:chi.square.density}, we have
        \begin{equation}\label{eq:LLT.beginning}
            \log\big(f_{r,\lambda}(x)\big) = - \log 2 - \frac{x + \lambda}{2} + \left(\frac{r}{2} - 1\right) \log\left(\frac{x}{2}\right) + \log \sum_{j=0}^{\infty} \frac{(\lambda x / 4)^j}{j! \, \Gamma(r/2 + j)}.
        \end{equation}
        Since
        \begin{equation}
            x = (r + \lambda) \, \left(1 + \frac{\delta_x}{D_{r,\lambda}}\right),
        \end{equation}
        we can rewrite \eqref{eq:LLT.beginning} as
        \begin{equation}\label{eq:LLT.beginning.next.1}
            \begin{aligned}
                \log\big(f_{r,\lambda}(x)\big)
                &= - \log 2 - \frac{r + \lambda}{2} \left(1 + \frac{\delta_x}{D_{r,\lambda}}\right) - \frac{\lambda}{2} + \left(\frac{r}{2} - 1\right) \log \left(\frac{r + \lambda}{2}\right) \\
                &\quad+ \left(\frac{r}{2} - 1\right) \log\left(1 + \frac{\delta_x}{D_{r,\lambda}}\right) - \log \Gamma\left(\frac{r}{2}\right) \\
                &\quad+ \log \sum_{j=0}^{\infty} \frac{1}{j!} \left(\frac{\lambda}{2} \, \bigg(1 + \frac{\delta_x}{D_{r,\lambda}}\bigg)\right)^j \, \frac{2^{-j} (r + \lambda)^j \Gamma(r/2)}{\Gamma(r/2 + j)}.
            \end{aligned}
        \end{equation}
        Using the expansions
        \begin{equation}
            \left(\frac{r}{2} - 1\right) \log \left(\frac{r + \lambda}{2}\right) = \left(\frac{r}{2} - 1\right) \log\left(\frac{r}{2}\right) + \frac{\lambda}{2} - \frac{\lambda + \lambda^2/4}{r} + \OO\left(\frac{1 \vee \lambda^3}{r^2}\right),
        \end{equation}
        and
        \begin{equation}
            \log \Gamma\left(\frac{r}{2}\right) = \left(\frac{r}{2} - \frac{1}{2}\right) \log \left(\frac{r}{2}\right) - \frac{r}{2} + \frac{1}{2} \log (2\pi) + \frac{1}{6 r} + \OO(r^{-3}),
        \end{equation}
        (the first one is just a Taylor expansion for $r^{-1}$ at $0$, and the second one can be found, for example, in \cite[p.257]{MR0167642})
        we can rewrite \eqref{eq:LLT.beginning.next.1} as
        \begin{equation}\label{eq:LLT.beginning.next.2}
            \begin{aligned}
                \log\big(f_{r,\lambda}(x)\big)
                &= - \frac{1}{2} \log (2 \pi \, 2 r) - \frac{r}{2} \, \frac{\delta_x}{D_{r,\lambda}} - \frac{\lambda}{2} \left(1 + \frac{\delta_x}{D_{r,\lambda}}\right) \\
                &\quad+ \left(\frac{r}{2} - 1\right) \log\left(1 + \frac{\delta_x}{D_{r,\lambda}}\right) - \frac{\frac{1}{6} + \lambda + \lambda^2/4}{r} + \OO\left(\frac{1 \vee \lambda^3}{r^2}\right) \\
                &\quad+ \log \sum_{j=0}^{\infty} \frac{1}{j!} \left(\frac{\lambda}{2} \, \bigg(1 + \frac{\delta_x}{D_{r,\lambda}}\bigg)\right)^j \, \frac{2^{-j} (r + \lambda)^j \Gamma(r/2)}{\Gamma(r/2 + j)}.
            \end{aligned}
        \end{equation}
        Now,
        \begin{equation}
            \begin{aligned}
                \frac{2^{-j} (r + \lambda)^j \Gamma(r/2)}{\Gamma(r/2 + j)}
                &= 1 + \frac{j (\lambda + 1) - j^2}{r} \\
                &\quad+ \OO\bigg(\frac{(1 \vee \lambda^2) (j + j^2) + (1 \vee \lambda) j^3 + j^4}{r^2}\bigg),
            \end{aligned}
        \end{equation}
        and, for any $a\in \R$,
        \begin{equation}
            \begin{aligned}
                &\log\left(\sum_{j=0}^{\infty} \frac{a^j}{j!} \, \left\{1 + \frac{j (\lambda + 1) - j^2}{r} + \OO\bigg(\frac{(1 \vee \lambda^2) (j + j^2) + (1 \vee \lambda) j^3 + j^4}{r^2}\bigg)\right\}\right) \\
                &= a + \log \left(1 + \frac{a (\lambda - a)}{r} + \OO\left(\frac{|a|}{r^2} \cdot \left\{\hspace{-1mm}
                    \begin{array}{l}
                        (1 \vee \lambda^2) (1 + |a|) \\
                        + (1 \vee \lambda) (1 + |a|^2) + (1 + |a|^3)
                    \end{array}
                    \hspace{-1mm}\right\}\right)\right),
            \end{aligned}
        \end{equation}
        so if we take
        \begin{equation}
            a = \frac{\lambda}{2} \, \bigg(1 + \frac{\delta_x}{D_{r,\lambda}}\bigg),
        \end{equation}
        we can rewrite \eqref{eq:LLT.beginning.next.2} as
        \begin{equation}\label{eq:LLT.beginning.next.3}
            \begin{aligned}
                \log\big(f_{r,\lambda}(x)\big)
                &= - \frac{1}{2} \log (2\pi \, 2r) - \frac{r}{2} \, \frac{\delta_x}{D_{r,\lambda}} \\[-1mm]
                &\quad+ \left(\frac{r}{2} - 1\right) \log\left(1 + \frac{\delta_x}{D_{r,\lambda}}\right) - \frac{\frac{1}{6} + \lambda + \lambda^2/4}{r} + \OO\left(\frac{1 \vee \lambda^3}{r^2}\right) \\
                &\quad+ \log \left(1 + \frac{\lambda}{2 r} \, \left(1 + \frac{\delta_x}{D_{r,\lambda}}\right) \left[\lambda - \frac{\lambda}{2} \, \left(1 + \frac{\delta_x}{D_{r,\lambda}}\right)\right] + \OO\left(\frac{1 \vee \lambda^4}{r^2}\right)\right).
            \end{aligned}
        \end{equation}
        Using the Taylor expansion
        \begin{equation}\label{eq:Taylor.log.1.plus.y}
            \log(1 + y) = y - \frac{y^2}{2} + \frac{y^3}{3} - \frac{y^4}{4} + \frac{y^5}{5} + \OO_{\eta}(y^6),
        \end{equation}
        valid for $|y| \leq \eta < 1$, we have
        \begin{align}\label{eq:LLT.beginning.next.4}
            \log\big(f_{r,\lambda}(x)\big)
            &= - \frac{1}{2} \log (2\pi \, 2r) - \frac{r}{2} \, \frac{\delta_x}{D_{r,\lambda}} - \frac{\frac{1}{6} + \lambda}{r} + \OO\left(\frac{(1 \vee \lambda^4) + \lambda^2 |\delta_x|^2}{r^2}\right) \notag \\
            &\quad+ \left(\frac{r}{2} - 1\right) \, \left\{\hspace{-1mm}
                \begin{array}{l}
                    \frac{\delta_x}{D_{r,\lambda}} - \frac{1}{2} \big(\frac{\delta_x}{D_{r,\lambda}}\big)^2 + \frac{1}{3} \big(\frac{\delta_x}{D_{r,\lambda}}\big)^3 - \frac{1}{4} \big(\frac{\delta_x}{D_{r,\lambda}}\big)^4 \\
                    + \frac{1}{5} \big(\frac{\delta_x}{D_{r,\lambda}}\big)^5 + \OO_{\eta}\left(\frac{1 + |\delta_x|^6}{r^3}\right)
                \end{array}
                \hspace{-1mm}\right\} \notag \\[0.5mm]
            &= - \frac{1}{2} \log (2\pi \, 2r) - \frac{\frac{1}{6} + \lambda}{r} + \OO\left(\frac{(1 \vee \lambda^4) + \lambda^2 |\delta_x|^2}{r^2}\right) \notag \\[0.5mm]
            &\quad- \frac{\sqrt{2 (r + 2\lambda)}}{(r + \lambda)} \delta_x - \frac{(r - 2) (r + 2\lambda)}{2 (r + \lambda)^2} \delta_x^2 + \frac{(r - 2) 2^{1/2} (r + 2\lambda)^{3/2}}{3 (r + \lambda)^3} \delta_x^3 \notag \\
            &\quad- \frac{(r - 2) (r + 2\lambda)^2}{2 (r + \lambda)^4} \delta_x^4 + \frac{(r - 2) 2^{3/2} (r + 2\lambda)^{5/2}}{5 (r + \lambda)^5} \delta_x^5.
        \end{align}
        Since
        \begin{equation}
            \begin{aligned}
                - \frac{\sqrt{2 (r + 2\lambda)}}{(r + \lambda)} &= - \frac{\sqrt{2}}{r^{1/2}} + \OO\left(\frac{1 \vee \lambda^2}{r^{5/2}}\right), \\[0.5mm]
                - \frac{(r - 2) (r + 2\lambda)}{2 (r + \lambda)^2} &= - \frac{1}{2} + \frac{1}{r} + \OO\left(\frac{1 \vee \lambda^2}{r^2}\right), \\[0.5mm]
                \frac{(r - 2) 2^{1/2} (r + 2\lambda)^{3/2}}{3 (r + \lambda)^3} &= \frac{\sqrt{2} / 3}{r^{1/2}} - \frac{2^{3/2}/3}{r^{3/2}} + \OO\left(\frac{1 \vee \lambda^2}{r^{5/2}}\right), \\[0.5mm]
                - \frac{(r - 2) (r + 2\lambda)^2}{2 (r + \lambda)^4} &= - \frac{1}{2 r} + \OO\left(\frac{1}{r^2}\right), \\[0.5mm]
                \frac{(r - 2) 2^{3/2} (r + 2\lambda)^{5/2}}{5 (r + \lambda)^5} &= \frac{2^{3/2}/5}{r^{3/2}} + \OO\left(\frac{1}{r^{5/2}}\right),
            \end{aligned}
        \end{equation}
        and
        \begin{equation}
            - \frac{1}{2} \log (2\pi \, 2r) = -\frac{1}{2} \log (2 \pi \, 2 (r + 2\lambda)) + \frac{\lambda}{r} + \OO\left(\frac{1 \vee \lambda^2}{r^2}\right),
        \end{equation}
        we can rewrite \eqref{eq:LLT.beginning.next.4} as
        \begin{align}\label{eq:lem:LLT.chi.square.log.eq.in.proof}
            \log\bigg(\frac{f_{r,\lambda}(x)}{\frac{1}{\sqrt{2 (r + 2\lambda)}} \phi(\delta_x)}\bigg)
            &= r^{-1/2} \, \bigg\{\frac{\sqrt{2}}{3} \, \delta_x^3 - \sqrt{2} \, \delta_x\bigg\} + r^{-1} \, \bigg\{-\frac{1}{2} \, \delta_x^4 + \delta_x^2 - \frac{1}{6}\bigg\} \notag \\[-2mm]
            &\quad+ r^{-3/2} \, \bigg\{\frac{2^{3/2}}{5} \, \delta_x^5 - \frac{2^{3/2}}{3} \, \delta_x^3\bigg\} + \OO\left(\frac{(1 \vee \lambda^4) + \lambda^2 |\delta_x|^2}{r^2}\right),
        \end{align}
        which proves \eqref{eq:lem:LLT.chi.square.eq.log}.
        To obtain \eqref{eq:lem:LLT.chi.square.eq} and conclude the proof, we take the exponential on both sides of the last equation and we expand the right-hand side with
        \begin{equation}\label{eq:Taylor.exponential}
            e^y = 1 + y + \frac{y^2}{2} + \frac{y^3}{6} + \OO(e^{\widetilde{\eta}} y^4), \quad -\infty < y \leq \widetilde{\eta}.
        \end{equation}
        For $r$ large enough and uniformly for $x\in B_{r,\lambda}(\eta)$, the right-hand side of \eqref{eq:lem:LLT.chi.square.log.eq.in.proof} is $\OO_{\lambda}(1)$.
        When this bound is taken as $y$ in \eqref{eq:Taylor.exponential}, it explains the error in \eqref{eq:lem:LLT.chi.square.eq}.
    \end{proof}

    \begin{proof}[Proof of Theorem~\ref{thm:refined.approximations}]
        Let
        \begin{equation}
            c = d_1 + \frac{d_2}{\sqrt{r}} + \frac{d_3}{r},
        \end{equation}
        where $d_1,d_2,d_3\in \R$ are to be chosen later, then we have the Taylor expansion
        \begin{equation}\label{eq:c}
            \begin{aligned}
                \int_{\delta_{a-c}}^{\delta_a} \phi(y) \rd y
                &= \phi(\delta_a) \int_{\delta_{a-c}}^{\delta_a} \rd y + \phi'(\delta_a) \int_{\delta_{a-c}}^{\delta_a} (y - \delta_a) \rd y + \frac{\phi''(\delta_a)}{2} \int_{\delta_{a-c}}^{\delta_a} (y - \delta_a)^2 \rd y \\
                &\quad+ \OO\left(\frac{\phi'''(\delta_a)}{6} \int_{\delta_{a-c}}^{\delta_a} (y - \delta_a)^3 \rd y\right) \\
                &= \phi(\delta_a) \, \left\{\hspace{-1mm}
                \begin{array}{l}
                    \frac{c}{\sqrt{2 (r + 2\lambda)}} + \frac{c^2 \, \delta_a}{4 (r + 2\lambda)} \\[2mm]
                    + \frac{c^3 (\delta_a^2 - 1)}{6 \cdot 2^{3/2} (r + 2 \lambda)^{3/2}} + \OO\big(\frac{1 + |\delta_a|^3}{r^2}\big)
                \end{array}
                \hspace{-1mm}\right\} \\
                &= \phi(\delta_a) \, \left\{\hspace{-1mm}
                \begin{array}{l}
                    r^{-1/2} \, \frac{1}{\sqrt{2}} \, d_1 + r^{-1} \, \big(\frac{\delta_a}{4} \, d_1^2 + \frac{1}{\sqrt{2}} \, d_2\big) \\[1mm]
                    r^{-3/2} \, \bigg(\hspace{-1mm}
                        \begin{array}{l}
                            \frac{(\delta_a^2 - 1)}{12 \sqrt{2}} \, d_1^3 - \frac{\lambda}{\sqrt{2}} \, d_1 \\[1mm]
                            + \frac{\delta_a}{2} \, d_1 d_2 + \frac{1}{\sqrt{2}} \, d_3
                        \end{array}
                        \hspace{-1mm}\bigg) + \OO\big(\frac{1 + |\delta_a|^3}{r^2}\big)
                \end{array}
                \hspace{-1mm}\right\}
            \end{aligned}
        \end{equation}
        We also have the straightforward large deviation bounds
        \begin{equation}\label{eq:LD}
            \begin{aligned}
                &\int_{[a,\infty) \cap B_{r,\lambda}^{\hspace{0.2mm}c}(1/2)} \hspace{-0.6mm} f_{r,\lambda}(x) \rd x = \OO(e^{-\beta r^{1/3}}), \\
                &\int_{[a,\infty) \cap B_{r,\lambda}^{\hspace{0.2mm}c}(1/2)} \hspace{-0.6mm} \phi(y) \rd y = \OO(e^{-\beta r^{1/3}}),
            \end{aligned}
        \end{equation}
        where $\beta = \beta(\lambda) > 0$ is a small enough constant, and the local approximation in Lemma~\ref{lem:LLT.chi.square} yields
        \begin{equation}\label{eq:Cressie.generalization.eq.4}
            \begin{aligned}
                \int_a^{\infty} \hspace{-0.6mm} f_{r,\lambda}(x) \rd x - \int_{\delta_a}^{\infty} \hspace{-1mm} \phi(y) \rd y
                &= r^{-1/2} \bigg\{\frac{\sqrt{2}}{3} \Psi_3(\delta_a) - \sqrt{2} \Psi_1(\delta_a)\bigg\} \\
                &\quad+ r^{-1} \left\{\frac{1}{9} \Psi_6(\delta_a) - \frac{7}{6} \Psi_4(\delta_a) + 2 \Psi_2(\delta_a) - \frac{1}{6} \Psi(\delta_a)\right\} \\
                &\quad+ r^{-3/2} \left\{\hspace{-1mm}
                    \begin{array}{l}
                        \frac{\sqrt{2}}{81} \Psi_9(\delta_a) - \frac{5}{9\sqrt{2}} \Psi_7(\delta_a) + \frac{47}{15\sqrt{2}} \Psi_5(\delta_a) \\[1.5mm]
                        - \frac{37}{9\sqrt{2}} \Psi_3(\delta_a) + \frac{1}{3\sqrt{2}} \Psi_1(\delta_a)
                    \end{array}
                    \hspace{-1mm}\right\} \\
                &\quad+ \OO\left(\frac{1 \vee \lambda^4}{r^2}\right),
            \end{aligned}
        \end{equation}
        where $\Psi_k(\delta_a) \leqdef \int_{\delta_a} y^k \phi(y) \rd y$.
        Now, using the fact that
        \begin{equation}
            \begin{aligned}
                &\Psi_9(\delta_a) = (384 + 192 \delta_a^2 + 48 \delta_a^4 + 8 \delta_a^6 + \delta_a^8) \phi(\delta_a), \\
                &\Psi_7(\delta_a) = (48 + 24 \delta_a^2 + 6 \delta_a^4 + \delta_a^6) \phi(\delta_a), \\
                &\Psi_6(\delta_a) = (15 \delta_a + 5 \delta_a^3 + \delta_a^5) \phi(\delta_a) + 15 \Psi(\delta_a), \\
                &\Psi_5(\delta_a) = (8 + 4 \delta_a^2 + \delta_a^4) \phi(\delta_a), \\
                &\Psi_4(\delta_a) = (3 \delta_a + \delta_a^3) \phi(\delta_a) + 3 \Psi(\delta_a), \\
                &\Psi_3(\delta_a) = (2 + \delta_a^2) \phi(\delta_a), \\
                &\Psi_2(\delta_a) = \delta_a \phi(\delta_a) + \Psi(\delta_a), \\
                &\Psi_1(\delta_a) = \phi(\delta_a),
            \end{aligned}
        \end{equation}
        where $\Psi$ denotes the survival function of the standard normal distribution, Equations~\eqref{eq:c}, \eqref{eq:LD} and \eqref{eq:Cressie.generalization.eq.4} together yield
        \begin{equation}\label{eq:Cressie.generalization.eq.5}
            \begin{aligned}
                &\int_a^{\infty} \hspace{-0.6mm} f_{r,\lambda}(x) \rd x - \int_{\delta_{a - c}}^{\infty} \hspace{-1mm} \phi(y) \rd y \\
                &\qquad= r^{-1/2} \bigg\{\frac{\sqrt{2}}{3} (\delta_a^2 - 1) - \frac{1}{\sqrt{2}} \, d_1\bigg\} \, \phi(\delta_a) \\
                &\qquad\quad+ r^{-1} \left\{\frac{\delta_a}{18} \left(2 \delta_a^4 - 11 \delta_a^2 + 3\right) - \left(\frac{\delta_a}{4} \, d_1^2 + \frac{1}{\sqrt{2}} \, d_2\right)\right\} \, \phi(\delta_a) \\
                &\qquad\quad+ r^{-3/2} \left\{\hspace{-1mm}
                    \begin{array}{l}
                        \frac{\sqrt{2}}{81} \delta_a^8 - \frac{29}{81\sqrt{2}} \delta_a^6 + \frac{133}{135\sqrt{2}} \delta_a^4 - \frac{23}{135 \sqrt{2}} \delta_a^2 - \frac{1}{135 \sqrt{2}} \\[1mm]
                        - \left(\frac{(\delta_a^2 - 1)}{12 \sqrt{2}} \, d_1^3 - \frac{\lambda}{\sqrt{2}} \, d_1 + \frac{\delta_a}{2} \, d_1 d_2 + \frac{1}{\sqrt{2}} \, d_3\right)
                    \end{array}
                    \hspace{-1mm}\right\} \, \phi(\delta_a) \\
                &\qquad\quad+ \OO\left(\frac{1 \vee \lambda^4}{r^2}\right).
            \end{aligned}
        \end{equation}
        If we select $d_1 = d_2 = d_3 = 0$, then
        \begin{equation}\label{eq:Cressie.generalization.eq.5.order.0}
            \begin{aligned}
                &\max_{a\in \R} \left|\int_a^{\infty} \hspace{-0.6mm} f_{r,\lambda}(x) \rd x - \int_{\delta_{a - c}}^{\infty} \hspace{-1mm} \phi(y) \rd y\right| \\[-0.5mm]
                &\quad\leq r^{-1/2} \, \max_{a\in \R} \frac{\sqrt{2}}{3} |\delta_a^2 - 1| \phi(\delta_a) + \OO(r^{-1}).
            \end{aligned}
        \end{equation}
        Since $\max_{y\in \R} \frac{\sqrt{2}}{3} |y^2 - 1| \phi(y) = \frac{1}{\sqrt{9 \pi}}$, this proves \eqref{eq:order.0.approx}.
        If we select $d_1 = \frac{2}{3} (\delta_a^2 - 1)$ and $d_2 = d_3 = 0$ to cancel the first brace in \eqref{eq:Cressie.generalization.eq.5}, then
        \begin{equation}\label{eq:Cressie.generalization.eq.5.order.1}
            \begin{aligned}
                &\max_{a\in \R} \left|\int_a^{\infty} \hspace{-0.6mm} f_{r,\lambda}(x) \rd x - \int_{\delta_{a - c}}^{\infty} \hspace{-1mm} \phi(y) \rd y\right| \\[-0.5mm]
                &\quad\leq r^{-1} \, \max_{a\in \R} \frac{1}{18} |7 \delta_a^3 - \delta_a| \phi(\delta_a) + \OO\left(\frac{1 \vee \lambda}{r^{3/2}}\right),
            \end{aligned}
        \end{equation}
        which proves \eqref{eq:order.1.approx}.
        If we select $d_1 = \frac{2}{3} (\delta_a^2 - 1)$, $d_2 = \frac{1}{9\sqrt{2}} (\delta_a - 7 \delta_a^3)$ and $d_3 = 0$ to cancel the first two braces in \eqref{eq:Cressie.generalization.eq.5}, then
        \begin{equation}\label{eq:Cressie.generalization.eq.5.order.2}
            \begin{aligned}
                &\max_{a\in \R} \left|\int_a^{\infty} \hspace{-0.6mm} f_{r,\lambda}(x) \rd x - \int_{\delta_{a - c}}^{\infty} \hspace{-1mm} \phi(y) \rd y\right| \\[-0.5mm]
                &\quad\leq r^{-3/2} \, \max_{a\in \R} \frac{1}{405\sqrt{2}} \left|219 \delta_a^4 + (270 \lambda - 14) \delta_a^2 - (270 \lambda + 13)\right| \phi(\delta_a) \\[0.5mm]
                &\quad\quad+ \OO\left(\frac{1 \vee \lambda^4}{r^2}\right),
            \end{aligned}
        \end{equation}
        which proves \eqref{eq:order.2.approx}.
        If we select $d_1 = \frac{2}{3} (\delta_a^2 - 1)$, $d_2 = \frac{1}{9\sqrt{2}} (\delta_a - 7 \delta_a^3)$ and $d_3 = \frac{1}{405} \big(219 \delta_a^4 + (270 \lambda - 14) \delta_a^2 - (270 \lambda + 13)\big)$ to cancel the three braces in \eqref{eq:Cressie.generalization.eq.5}, then
        \begin{equation}\label{eq:Cressie.generalization.eq.5.order.3}
            \max_{a\in \R} \left|\int_a^{\infty} \hspace{-0.6mm} f_{r,\lambda}(x) \rd x - \int_{\delta_{a - c}}^{\infty} \hspace{-1mm} \phi(y) \rd y\right| = \OO\left(\frac{1 \vee \lambda^4}{r^2}\right),
        \end{equation}
        which proves \eqref{eq:order.3.approx}.
        This ends the proof.
    \end{proof}

\section{Moments of the central and noncentral chi-square distribution}\label{sec:moments}

    In the lemma below, we prove a general formula for the central moments of the central and noncentral chi-square distribution, and we evaluate the first, second, third, fourth and sixth central moments explicitly. This lemma is used to estimate the $\asymp r^{-1}$ errors in \eqref{eq:estimate.I.begin} of the proof of Theorem~\ref{thm:total.variation}.
    It is also a preliminary result for the proof of Corollary~\ref{cor:central.moments.chi.square.on.events} below, where the central moments are estimated on various events.

    \begin{lemma}[Central moments]\label{lem:central.moments.chi.square}
        Let $X\sim \chi_r^2(\lambda)$ for some $r > 0$ and $\lambda\geq 0$.
        We have
        \begin{equation}\label{eq:lem:central.moments.chi.square.examples}
            \begin{aligned}
                &\EE[(X - (r + \lambda))] = 0, \\[0.5mm]
                &\EE[(X - (r + \lambda))^2] = 2 \, (r + 2 \lambda), \\[0.5mm]
                &\EE[(X - (r + \lambda))^3] = 8 \, (r + 3 \lambda), \\[0.5mm]
                &\EE[(X - (r + \lambda))^4] = 12 \, (r^2 + 4 r (1 + \lambda) + 4 \lambda (4 + \lambda)), \\
                &\EE[(X - (r + \lambda))^6] = 40 \left(\hspace{-1mm}
                    \begin{array}{l}
                        3 r^3 + 2 r^2 (26 + 9 \lambda) + 12 r (8 + 26 \lambda + 3 \lambda^2) \\
                        + 24 \lambda (24 + 18 \lambda + \lambda^2)
                    \end{array}
                    \hspace{-1mm}\right).
            \end{aligned}
        \end{equation}
    \end{lemma}

    \begin{proof}[Proof of Lemma~\ref{lem:central.moments.chi.square}]
        One way to compute these central moments would be to apply the recurrence formula developed in \cite{MR2411792}.
        An other method consists in differentiating the moment-generating function
        \begin{equation}
            \EE[e^{t X}] = (1 - 2 t)^{-r/2} \exp\bigg(\frac{\lambda t}{1 - 2 t}\bigg), \quad t < 1/2,
        \end{equation}
        in order to find $\EE[X^i], \, i\in \N,$ and then use the binomial formula:
        \begin{equation}
            \EE[(X - (r + \lambda))^n] = \sum_{i=0}^n \binom{n}{i} \EE[X^i] \, (-1)^{n-i} (r + \lambda)^{n-i}, \quad n\in \N.
        \end{equation}
        Using the latter approach with \texttt{Mathematica} give us the result.
    \end{proof}

    We can also estimate the moments of Lemma~\ref{lem:central.moments.chi.square} on various events.
    The corollary below is used to estimate the $\asymp r^{-1/2}$ errors in \eqref{eq:estimate.I.begin} of the proof of Theorem~\ref{thm:total.variation}.

    \begin{corollary}[Central moments on various events]\label{cor:central.moments.chi.square.on.events}
        Let $X\sim \chi_r^2(\lambda)$ for some $r > 0$ and $0 \leq \lambda \leq r$, and let $A\in \mathscr{B}(\R)$ be a Borel set.
        Then,
        \begin{equation}\label{eq:cor:central.moments.chi.square.on.events.main}
            \begin{aligned}
                &\big|\EE[(X - (r + \lambda)) \, \ind_{\{X\in A\}}]\big| \leq 6^{1/2} \, r^{1/2} \, (\PP(X\in A^c))^{1/2}, \\
                &\big|\EE[(X - (r + \lambda))^2 \, \ind_{\{X\in A\}}] - 2 \, (r + 2 \lambda)\big| \leq 348^{1/2} \, r \, (\PP(X\in A^c))^{1/2}, \\
                &\big|\EE[(X - (r + \lambda))^3 \, \ind_{\{X\in A\}}] - 8 \, (r + 3 \lambda)\big| \leq 61960^{1/2} \, r^{3/2} \, (\PP(X\in A^c))^{1/2}.
            \end{aligned}
        \end{equation}
    \end{corollary}

    \begin{proof}[Proof of Corollary~\ref{cor:central.moments.chi.square.on.events}]
        Note that \eqref{eq:lem:central.moments.chi.square.examples} implies
        \begin{equation}\label{eq:cor:central.moments.chi.square.on.events.proof.first}
            \begin{aligned}
                &\EE[(X - (r + \lambda))^2] = 6 \, r, \\
                &\EE[(X - (r + \lambda))^4] \leq 348 \, r^2, \\
                &\EE[(X - (r + \lambda))^6] \leq 61960 \, r^3.
            \end{aligned}
        \end{equation}
        By \eqref{eq:lem:central.moments.chi.square.examples}, we also have
        \begin{equation}
            \begin{aligned}
                &\big|\EE[(X - (r + \lambda)) \, \ind_{\{X\in A\}}]\big| = \big|\EE[(X - (r + \lambda)) \, \ind_{\{X\in A^c\}}]\big|, \\
                &\big|\EE[(X - (r + \lambda))^2 \, \ind_{\{X\in A\}}] - 2 \, (r + 2 \lambda)\big| = \big|\EE[(X - (r + \lambda))^2 \, \ind_{\{X\in A^c\}}]\big|, \\
                &\big|\EE[(X - (r + \lambda))^3 \, \ind_{\{X\in A\}}] - 8 \, (r + 3 \lambda)\big| = \big|\EE[(X - (r + \lambda))^3 \, \ind_{\{X\in A^c\}}]\big|.
            \end{aligned}
        \end{equation}
        We get \eqref{eq:cor:central.moments.chi.square.on.events.main} by applying the Cauchy-Schwarz inequality and bounding using \eqref{eq:cor:central.moments.chi.square.on.events.proof.first}.
    \end{proof}

\end{appendices}

\section*{Acknowledgments}

We thank Robert Ferydouni (University of California - Santa Cruz) for collecting some of the references in Section~\ref{sec:related.works} and helping us use the \texttt{latex2exp} package in \texttt{R}.
F.\ Ouimet is supported by postdoctoral fellowships from the NSERC (PDF) and the FRQNT (B3X supplement and B3XR).

%
%

\phantomsection
\addcontentsline{toc}{chapter}{References}



\end{document}